\documentclass[12pt]{amsart}
\usepackage{graphicx}
\usepackage{xcolor}
\usepackage{url}
\usepackage{amsmath, amssymb, mathtools, epic, eepic, floatflt, microtype}

\usepackage{graphpap,color,paralist,pstricks}
\usepackage[mathscr]{eucal}
\usepackage[pdftex,colorlinks,backref=page,citecolor=blue]{hyperref}
\usepackage{enumitem} 

\usepackage{euler, palatino}
\usepackage{fullpage}
\usepackage{amscd,amsthm}
\usepackage{comment}
\usepackage{color}
\usepackage{newtxmath}
\usepackage{mathtools}
\usepackage{geometry}

\usepackage[utf8]{inputenc}
\usepackage{amssymb}

\numberwithin{equation}{section}
\newtheorem{theorem}{Theorem}
\newtheorem{proposition}[theorem]{Proposition}

\newtheorem{question}[theorem]{Question}
\newtheorem{corollary}[theorem]{Corollary}
\newtheorem{lemma}[theorem]{Lemma}
\theoremstyle{definition}
\newtheorem{remark}[theorem]{Remark}
\newtheorem{example}[theorem]{Example}
\newtheorem{definition}[theorem]{Definition}

\begin{document}

\title{Shellability of the quotient order \\ on lattice path matroids}
\date{\today}

\author[C. Benedetti]{Carolina Benedetti}
\address[C. Benedetti]{Department
of Mathematics, University of the Andes, Bogot\'a, Colombia}
\email{c.benedetti@uniandes.edu.co}

\author[A. Dochtermann]{Anton Dochtermann}
\address[A. Dochtermann]{Department of Mathematics, 
Texas State University, 
San Marcos, TX, USA}
\email{dochtermann@txstate.edu}

\author[K. Knauer]{Kolja Knauer}
\address[K. Knauer]{
Department of Mathematics, University of Barcelona, Spain}
\email{kolja.knauer@ub.edu}

\author[Y. Li]{Yupeng Li}
\address[Y. Li]{Department
of Mathematics, Duke University, Durham, NC, USA}
\email{ypli@math.duke.edu}

\begin{abstract}
The concept of a matroid quotient has connections to fundamental questions in the geometry of flag varieties. In previous work, Benedetti and Knauer characterized quotients in the class of lattice path matroids (LPMs) in terms of a simple combinatorial condition. As a consequence they showed that the quotient order on LPMs yields a graded poset whose rank polynomial relates to a refinement of the Catalan numbers. 
In this work we show that this poset admits an EL-labeling, implying that the order complex is shellable and hence enjoys several combinatorial and topological properties. We use this to establish bounds on the M\"obius function of the poset, interpreting falling chains in the EL-labeling in terms of properties of underlying permutations.  Furthermore, we show that this EL-labeling is in fact a Whitney labeling, in the sense of the recent notion introduced by Gonz\'alez D'Le\'on and Hallam.
\end{abstract}
\maketitle

\section{Introduction}
Matroids are combinatorial objects that generalize the notion of independence, and can be studied from a variety of perspectives. In particular, any rank $k$ representable matroid $M$ (over a field ${\mathbb K}$) on ground set $E = [n]$ can be described as a $k$-dimensional linear subspace $L$ in ${\mathbb K}^n$.  If one thinks of $L$ as a point in the Grassmannian $\text{Gr}_{k,n}$, the torus closure of $L$ in this space yields a projective toric variety that is isomorphic to the toric variety obtained from the matroid basis polytope of $M$.  
In recent years, there has been increased interest in studying various notions of \emph{positivity} in geometry and combinatorics (where one restricts to the case ${\mathbb K} = {\mathbb R}$). In particular, the notion of a positroid has been introduced by Postnikov in \cite{positivity} as a way to stratify the totally nonnegative Grassmannian $Gr^{\geq 0}_{k,n}$.  Positroids are examples of representable matroids, and subsequent work has lead to a number of combinatorial characterizations.

The \emph{(real) full flag variety $\mathcal F\ell_n$} consists of all sequences $F:V_{0}\subset V_1\subset\cdots\subset V_{n}$ of $\mathbb R$-vector spaces where $\dim V_i=i$ for all $i$. Note that such a flag $F$ has that property that each matroid $M_i$ arising from $V_i$ is (by definition) representable. Such a collection of matroids $(M_0,M_1,\cdots,M_n)$ give rise to notion of a matroid quotient,  where each $M_i$ is a \emph{quotient} of $M_{j}$ for $1\leq i\leq j\leq n$. One can check that every circuit (i.e. minimally dependent set) in $M_j$ is a union of circuits in $M_i$. 
This leads to one of several equivalent ways to define a matroid quotient, and in Definition \ref{def:matroid_quotient} we provide another characterization.
A \emph{full flag matroid} is a sequence of matroids $(M_0,M_1,\cdots,M_n)$ such that $M_i$ is a quotient of $M_{j}$ for $1\leq i\leq j\leq n$. Hence, points in $\mathcal F\ell_n$ give rise to full flag matroids, although it is not the case that every full flag matroid comes from a point in $\mathcal F\ell_n$.

Lattice path matroids (LPMs), introduced by Bonin, de Mier, and Noy in \cite{BdMN2003}, are a particular class of positroids. For a fixed $n$, the data of a lattice path matroid $M = M[U,L]$ is given by an `upper' path $U$ and a `lower' path $L$, both lattice paths from $(0,0)$ to $(n-k,k)$ for some $k\leq n$. The bases of $M$ are given by all the lattice paths from $(0,0)$ to $(n-k,k)$ that lie in between $U$ and $L$. 
From a geometric perspective, an LPM corresponds to a generic point in a cell arising from the Richardson cell decomposition of the Grassmannian $Gr_{k,n}$.

In \cite{BenKna} Benedetti and Knauer considered quotients of LPMs that are themselves LPMs. They provided a combinatorial condition to identify when an LPM is a quotient of another LPM, see Theorem \ref{thm:lpm_quotients} below.
One can then define $\mathcal{P}_n = (\mathcal{P}_n, \leq_q)$ to be the poset whose elements are all LPMs on ground set $[n]$, where $M'\leq _q M$ if and only if $M'$ is a \emph{quotient} of $M$. 
From a matroidal point of view, a maximal chain $L_0\leq_q L_1\leq_q\cdots\leq_q L_n$ in $\mathcal{P}_n$ is a a full flag matroid. In this case we refer to maximal chains of $\mathcal{P}_n$ as (full) \emph{lattice path flag matroids}, or LPFMs for short. 

Now, the \emph{nonnegative (full) flag variety} $\mathcal F\ell_n^{\geq 0}$ consists of flags $F:V_{0}\subset V_1\subset\cdots\subset V_{n}\in\mathcal F\ell_n$ for which there exists a matrix $A_F$ that realizes $F$ with the additional property that for each $i$, the submatrix given by the top $i$ rows of $A_F$ has nonnegative maximal minors. From this it follows that the flag matroid $(M_0,M_1,\dots,M_n)$ arising from $F$ has the property that each $M_i$ is a positroid. If a flag of matroids $(M_0,M_1,\dots,M_n)$ has the property that each $M_i$ is a positroid, we refer to it as a \emph{positroid (full) flag}. 
We emphasize that it is \emph{not} the case that every positroid flag $M_0\leq_qM_1\leq_q\cdots\leq_q M_n$ corresponds to a point in $\mathcal F\ell_n^{\geq 0}$, see \cite[Example 7]{BenKna}. 

Given a flag of matroids $\mathcal M:M_0\leq_qM_1\leq_q\cdots\leq_q M_n$, the Minkowski sum of the matroid polytopes $P_{M_i}$ is the \emph{flag matroid polytope $P_{\mathcal M}$}. Such a polytope  $P_{\mathcal M}$ has vertices given by  certain permutations of the point $(1,2,\dots,n)$. A flag of positroids 
$\mathcal M:M_0\leq_qM_1\leq_q\cdots\leq_q M_n$ comes from a point in $\mathcal F\ell_n^{\geq 0}$ if and only if (the 1-skeleton of) $P_{\mathcal M}$ is an interval in the Bruhat order of the symmetric group $S_n$. Flag positroid polytopes arising points in $\mathcal F\ell_n^{\geq 0}$ are known as \emph{Bruhat interval polytopes}  \cite{KW13}. As proved in \cite[Corollary 33]{BenKna}, each LPFM has the property that its polytope is an interval in the Bruhat order, hence it is a Bruhat interval polytope. In particular, every LPFM is realized by a point in $\mathcal F\ell_n^{\geq 0}$.

\subsection{Our contributions}
In our manuscript we are interested in the shellability of the order complex $\Delta(\mathcal{P}_n)$ of $\mathcal{P}_n$.
In general, determining if a given poset $P$ is shellable is difficult (NP-complete in the appropriate sense, see \cite{NPcomplete}), and several methods for constructing shellings have been developed in recent years. Of particular relevance for us is a technique that involves labeling the edges of the Hasse diagram of $P$ with certain properties. The resulting \emph{EL-labelings} leads to a notion of a \emph{lexicographic shelling} of the poset, see below for details.

Shellings also have close connections to the \emph{M\"obius function} $\mu$ of a poset $P$.
In particular if $P$ admits an EL-labeling, then $\mu(x,y)$ is given by the number of \emph{falling chains} in the interval $[x,y]$ (see Section \ref{sec:shelling}). This leads to a way to compute the M\"obius function, and also provides a topological interpretation. From the definition of an EL-labeling it follows that every interval $[x,y]$ is shellable (and hence homotopy equivalent to a wedge of spheres), and $\mu(x,y)$ calculates the number of spheres in this wedge product.

Now we return to $\mathcal{P}_n = (\mathcal{P}_n, \leq_q)$, the quotient poset of LPMs on ground set $[n]$. The \emph{good pairing} from \cite{BenKna} gives rise to an edge labeling of the Hasse diagram on the poset defined by $\leq_q$. Our first result is the following.

\newtheorem*{thm:EL}{Theorem \ref{thm:EL}}
\begin{thm:EL}
The good pair labeling $\lambda$ provides an EL-labeling of ${\mathcal P}_n$.  
\end{thm:EL}

This answers a question posed by the first and third author in \cite{BenKna}, originally suggested by Bruce Sagan \cite{Sagan}. As discussed above, Theorem \ref{thm:EL} implies that the simplicial complex $\Delta(\mathcal{P}_n)$ is shellable, and hence its geometric realization is homotopy equivalent to a wedge of spheres. This also provides us with a tool to compute (or at least estimate) $\mu(\mathcal{P}_n)$, and our next results provide bounds on these values.

Recall that $\mu(\mathcal{P}_n)$ is given by counting falling chains in the relevant interval, determined by the EL-labeling. We use the fact that these falling chains can be encoded by an ordered pair $(\sigma, \tau)$ of permutations of $n$, which allows us to use language from symmetric groups. For example, the condition of being a \emph{good pair} can be phrased in terms of the \emph{Lehmer code} of the underlying permutations, see Theorem \ref{thm:goodpairs}. In addition, if we let $\sigma = w_0$ denote the longest permutation, one can check that any choice of $\tau$ yields a falling chain. This implies that $\mu(\mathcal{P}_n) \geq n!$ for any $n$. In fact $\mu(\mathcal{P}_3) = 6$ and $\mu(\mathcal{P}_4) = 25$.

We next investigate the structure of falling chains in ${\mathcal P}_n$ under our EL-labeling, with the goal of further understanding the growth rate as $n$ increases. We let ${\mathcal F}_{n}$ denote the set of falling chains in ${\mathcal P}_{n}$. We first show that ${\mathcal P}_{n-1}$ naturally sits inside ${\mathcal P}_{n}$ as a subposet, and analyze how falling chains can be extended. We prove in Corollary \ref{cor:lowerbound} that 
\[|{\mathcal F}_{n+1}| \geq (n+1) |{\mathcal F}_n|.\]

Next we study how elements of ${\mathcal F}_{n}$, described by a pair of permutations $(\sigma, \tau)$, can be understood in terms of the left Bruhat order as we change $\sigma$. Recall that if $\sigma = w_0$ we obtain $n!$ falling chains (the most we could hope for), whereas if $\sigma = \text{id}$, no falling chains are possible. We prove that this extends to a more general monotonicity property for falling chains. 
In what follows, for a permutation $\sigma \in S_n$, we let $C_\sigma$ denote the set of falling chains in ${\mathcal F}_{n}$ with first permutation $\sigma$. We use $\preceq_L$ to denote left Bruhat order.

\newtheorem*{thm:Bruhat}{Theorem \ref{thm:Bruhat}}
\begin{thm:Bruhat}
 Given two permutations $\sigma,\sigma'\in S_n$, we have $|C_\sigma| \leq |C_{\sigma'}|$ if $\sigma \preceq_L \sigma'$.
\end{thm:Bruhat}

It's not hard to show that $|C_{w_0}| = n!$, but computing the value of $|C_\sigma|$ for general $\sigma$ seems difficult. However, we are able to provide an explicit formula for the case $\sigma = s_iw_0$ (see Theorem \ref{thm:siw0}).  

In the last part of the paper, we consider further properties of our good pair labeling. Motivated by notions of \emph{Whitney duality}, Gonz\'alez D'Le\'{o}n and Hallam \cite{GH} introduced the notion of an \emph{EW-labeling} of a poset. The conditions are similar to the EL definition, with an extra rank two switching property, see Definition \ref{defn:EW}. As shown in \cite{GH}, if a poset $P$ admits an EW-labeling $\lambda$, one can construct a \emph{Whitney dual} poset $Q_{\lambda}(P)$ with the property that the first and second Whitney numbers of $P$ are exchanged.
In Proposition \ref{prop:EW}, we prove that our good pair labeling indeed satisfies these conditions.

The paper is organized as follows. In Section \ref{sec:prelim} we discuss basic notions of matroid quotients, lattice path matroids, as well as tools from EL-shellings of posets. In Section \ref{sec:proofs} we prove our main results, namely that our good pair labeling is an EL-shelling of ${\mathcal P}_n$. Here we also study the structure of falling chains under this labeling in terms of properties of the underlying permutations. In Section \ref{sec:whitney_dual} we prove that our labeling is Whitney, and provide an example of the Whitney dual arising from ${\mathcal P}_3$. We end in Section \ref{sec:further} with some open questions.

\section{Preliminaries}\label{sec:prelim}

In this section, we provide definitions and set up notation used throughout the paper. We also present some preliminary results necessary to prove our main theorems. This includes basics of lattice path matroids and their quotients, as well as important facts from the theory of shellings and EL-labeling on posets.

\subsection{Matroids and quotients}

We first recall some relevant notions from matroid theory. We assume the reader is familiar with basic properties properties of matroids, we refer to \cite{Oxley} for any undefined terms. 

\begin{definition}\label{def:matroid_quotient}\hfill
\begin{enumerate}
    \item A \textit{matroid} $M$ on $E$ is a nonempty collection $\mathcal{B}(M)$ of subsets of $E$ satisfying the exchange axiom: for any $B_1,B_2\in\mathcal{B}(M)$ and $x\in B_1\backslash B_2$, there exists $y\in B_2\backslash B_1$ such that $(B_1\backslash \{x\})\cup \{y\}\in \mathcal{B}(M)$. Elements of $\mathcal{B}$ are called \textit{bases}.
    \item Three operations on matroids below are necessary to define matroid quotients:
    \begin{enumerate}
        \item Deletion: $M\backslash T$ is a matroid on $E\backslash T$ with bases
        \[\mathcal{B}(M\backslash T) = \text{ maximal members of }\{B-T:B\in \mathcal{B}(M)\}
        \]
        \item Restriction: $M|_{T}$ is a matroid on $T$ and defined by $M|_{T} = M\backslash(E\backslash T)$. We denote basis elements of $M|_{T}$ as $B_T$ 
        \item Contraction: $M/T$ is a matroid on $E\backslash T$ with bases 
        \[\mathcal{B}(M/T) = \{S\subseteq E\backslash T: S\cup B_T \in \mathcal{B}(M)\}\]
    \end{enumerate}
    \item  A matroid $M^\prime$ is a \textit{quotient} of $M$ if there exists a matroid $N$ on the ground set $E$ and some $T\subseteq E$ such that $M = N\backslash T$ and $M^\prime = N/T$.
\end{enumerate}
\end{definition}

The main objects that we study is lattice path matroids (LPMs). The data of an LPM is given by a pair of noncrossing monotone lattice paths $U$ and $L$, which start and end at the same vertex. This is typically depicted as a diagram in the plane grid as in Figure~\ref{fig:xmpl}, which is bounded above by $U$ and below by $L$. Any monotone lattice path from the bottom left to the upper right corner inside this diagram is identified with a set $B$, where $i\in B$ if and only if the $i$th step of the path is north. Now, the collection $\mathcal{B}$ of these sets forms the set of bases of a matroid called \textit{lattice path matroid}, which we denote as $M[U,L]$. As a special case, a matroid is \emph{Schubert} if it is an LPM of the form $M[U,L]$, where $U$ does all its North steps first. In particular, \emph{uniform} matroids are LPMs where furthermore $L$ does all its North steps last.  

One can see that LPMs are special cases of \emph{transversal} matroids as well as \emph{positroids}. We refer the reader to ~\cite{BdMN2003} for fundamental properties of LPMs.

\begin{figure}[htp]
    \centering
    \includegraphics[width=.3\textwidth]{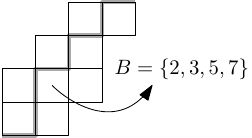}
    \caption{A basis in the diagram representing the LPM $M[1246,3568]$.}\label{fig:xmpl}
\end{figure}

Let $\mathcal{P}_n = (\mathcal{P}_n, \leq_q)$ be the partial order on the set of LPMs on ground set $[n]$, where $M' \leq _q M$ if $M'=M[U',L']$ is a quotient of $M=M[U,L]$. As observed in \cite{BenKna}, quotients of LPMs that are themselves LPMs have a more straightforward combinatorical description. For this, we need the following notion.

\begin{definition}\label{def:good_pair}
Let $M=M[U,L]$ be an LPM where $U=\{u_1,\cdots, u_k \}$, $L=\{\ell_1,\cdots ,\ell_k \}$. Let $1\leq i,j\leq k$. We say that $(\ell_i,u_j)$ is a \emph{good pair of $M$} if \begin{enumerate}
\item $i\leq j$,
\item $u_j-\ell_i\leq j-i$.
\end{enumerate}
     
  \end{definition}
               
Graphically, being a good pair can be visualized as follows. The step $u_j$ is such that its northern vertex $(a,b)$ determines the closed region $R_{u_j}$ bounded below by $L$, and lies in the halfspaces $x\geq a$ and $y\leq b$. Then the pair  $(\ell_i,u_j)$ is a good pair if  $\ell_i$ lies in $R_{u_j}$.  A pair $(\ell_i,u_j)$ that is not good is said to be \emph{bad}. Figure~\ref{fig:nonquotient} depicts a bad pair $(\ell_i,u_j)$. 

\begin{figure}[htp]
    \centering
    \includegraphics[width=.35\textwidth]{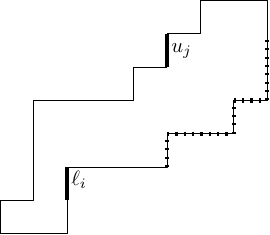}
    \caption{An LPM with a bad pair $(\ell_i, u_j)$. Exactly those upward segments $\ell\in L$ on the dotted path yield good pairs with $u_j$.}\label{fig:nonquotient}
\end{figure}

For $\{u_{1},\dots,u_{z} \}\subseteq U$ and $\{\ell_{1},\dots,\ell_{z} \}\subseteq L$, we call a sequence $((\ell_{i_1},u_{j_1}),\dots, (\ell_{i_z},u_{j_z}))$ a \emph{pairing of $M = M[U,L]$}. We say that a pairing of $M$ is \emph{good} if  $(\ell_{i_r},u_{j_r})$ is a good pair of $M[U',L']$ where $U'=U\setminus\{u_{j_1},\dots,u_{j_{r-1}} \}$ and $L'=L\setminus\{\ell_{j_1},\dots,\ell_{j_{r-1}} \}$, for $1\leq r\leq z-1$. Now let $M=M[U,L]$ and $M'=M[U',L']$ be LPMs over $[n]$ such that $L'\subseteq L,U'\subseteq U$ and set
$U\setminus U'=\{u_{i_1}<\ldots <u_{i_z}\}$, $L\setminus L'=\{\ell_{j_1}<\ldots < \ell_{j_z}\}$.
We refer to the sequence  $((\ell_{j_1},u_{i_1}),\ldots, (\ell_{j_z},u_{i_z}))$ as \emph{the greedy pairing of $(L\setminus L',U\setminus U')$}. From \cite{BenKna}, we have following characterization of quotients of LPMs.

\begin{theorem}\cite[Theorem 19]{BenKna}\label{thm:lpm_quotients}
Let $M=M[U,L]$  and $M'=M[U',L']$ be LPMs on the ground set $[n]$. We have that $M' \leq_q M$ if and only if $U'\subseteq U$, $L'\subseteq L$ and the greedy pairing of $(L\setminus L',U\setminus U')$ is good.
\end{theorem}

One part of the proof of Theorem \ref{thm:lpm_quotients} is the following result, which will also be useful for us.

\begin{lemma}\cite[Lemma 18]{BenKna}\label{lem:removing}
 Let $M=M[U,L]$, $\ell_i<\ell_{i'}\in L$ and $u_j<u_{j'}\in U$. If $(\ell_i,u_j)$ and $ (\ell_{i'},u_{j'})$ are good then $(\ell_{i},u_{j})$ is good in $M[U\setminus\{u_{j'}\},L\setminus\{\ell_{i'}\}]$ and $(\ell_{i'},u_{j'})$ is good in $M[U\setminus\{u_{j}\},L\setminus\{\ell_{i}\}]$.
\end{lemma}

Theorem~\ref{thm:lpm_quotients} along with Lemma~\ref{lem:removing} sheds some light on the structure of chains in an interval $[M',M]_{\leq q}$ in the poset $\mathcal P_n$. In particular,  if $((\ell_{i_1},u_{i_1}), \ldots, (\ell_{i_z},u_{j_z}))$ is the greedy pairing on  $(U\setminus U',L\setminus L')$ then any permutation of the set of pairs $\{(\ell_{i_1},u_{j_1}), \ldots, (\ell_{i_z},u_{j_z})\}$ gives rise to a sequence that is a good pairing. That is, every such permutation corresponds to a saturated chain in the interval $[M',M]_q$.  However, not all saturated chains arise this way.

\subsection{Shelling and EL-labelings} \label{sec:shelling}

To every poset $P$, one can associate an abstract simplicial complex $\Delta(P)$ called the \emph{order complex} of $P$. The vertices of $\Delta(P)$ are the elements of $P$ and the faces of $\Delta(P)$ are the chains (i.e., totally ordered subsets) of $P$.
Recall that a pure $d$-dimensional simplicial complex $\Delta$ is \emph{shellable} if there exists an ordering of its facets $F_1, \dots, F_t$ such that for all $j = 1, \dots t-1$ the complex
\[(\bigcup_{i=1}^j F_i) \cap F_{j+1}\]
\noindent
is pure and has dimension $d-1$. 

We will be interested in shelling simplicial complexes $\Delta = \Delta(P)$ that arise as the order complex of a poset $P$.  In this case the facets of $\Delta$ are given by the set ${\mathcal M}$ of maximal chains of $P$. There is a well-developed theory of `lexicographic shellings' via a labeling of $P$.

More precisely, for a poset $P$ an \emph{edge-labeling} is a map $\lambda:E(H(P)) \rightarrow \Lambda$ from the edge set of the Hasse diagram of $P$ to some poset $\Lambda$. Given such a labeling $\Lambda$, a saturated chain $C = (x_0 \lessdot x_1 \lessdot \cdots \lessdot x_k)$ of length $k$ gives rise to a $k$-tuple $\lambda(C) = (\lambda(x_0 \lessdot x_1), \dots, \lambda(x_{k-1}\lessdot x_k))$. We say that $C$ is \emph{weakly increasing} if $(\lambda(C))_i \leq_\Lambda (\lambda(C))_j$ whenever $i \leq j$. With this edge labeling, the \emph{lexicographic order} on unrefinable chains of length $k$ is defined by $C <_L C^\prime$ if there exists a $j \leq k$ such that $\lambda(C)_i = \lambda(C^\prime)_i$ for all $i < j$ and $\lambda(C)_j <_\Lambda \lambda(C^\prime)_j$.

\begin{definition}\label{def:EL}
Suppose $P$ is a poset with edge-labeling $\lambda:E(H(P)) \rightarrow \Lambda$. Then $\lambda$ is an $\emph{EL-labeling}$ if for every $x \leq y \in P$ the interval $[x,y]$ has the property that:
\begin{itemize}
    \item There exists a unique weakly increasing maximal chain $C$ in $[x,y]$;
    \item $C \leq_L C^\prime$ for all other maximal chains $C^\prime$ in $[x,y]$.
    \end{itemize}
\end{definition}

Work of Bj\"orner and Wachs \cite[Theorem 3.3]{BjoWac} implies that if a poset $P$ admits an EL-shelling then the complex $\Delta= \Delta(P)$ is shellable. In fact any linear ordering of the set ${\mathcal M}$ of maximal chains that extends the lexicographic order induced by $\lambda$ is a shelling order. 

EL-labelings have been especially effective in determining shellability of various classes of lattices and other posets.  For example, upper semimodular lattices \cite{Garsia}, geometric lattices \cite{Stanley74}, geometric semilattices \cite{WachsWalker}, supersolvable lattices \cite{Stanley72}, subgroup lattices of solvable groups \cite{Woodroofe}, and Bruhat order \cite{BjoWac} are all known to be lexicographically shellable. A number of these classes of lattices fit into the class of \emph{(co)modernistic} lattice introduced in \cite{ShweigWoodroofe}, which were shown to admit an EL-labeling in \cite{Li}.

Furthermore, one can use an EL-shelling of $P$ to compute the M\"obius function of $P$.  For this suppose $P$ is a poset of rank $n+1$ and that $\lambda$ is an EL-shelling of $P$. A maximal chain $m: \hat{0} = x_0 \lessdot x_1 \lessdot \dots \lessdot x_n = \hat{1}$ in $P$ is called a \emph{falling chain} if $\lambda(x_{i-1}\lessdot x_i)\nless \lambda(x_i\lessdot x_{i+1})$ for all $0 < i < n$.

\begin{theorem}\cite[Theorem 3.4]{BjoWac}\label{thm:falling}
Suppose a poset $P$ admits an EL-shelling $\lambda$. We have
\[(-1)^{n+1}\mu(\hat{0}, \hat{1}) = |\{m \in {\mathcal M}:\text{$m$ is a falling chain}\}|.\]
\end{theorem}

Note that by definition an EL-shelling of $P$ restricts to an EL-shelling of any interval of $P$, thought of as a subposet.  Hence we get a way to compute the entire M\"obius function of $P$ with the given EL-labeling $\Lambda$.

\section{Shellings and and M\"obius functions of \texorpdfstring{$\mathcal{P}_n$}{P}} \label{sec:proofs}

In this section we prove Theorem \ref{thm:EL}, and in particular describe an edge labeling of the poset $\mathcal{P}_n$. We then discuss implications for the M\"obius function of $\mathcal{P}_n$ and also show that our label gives rise to a Whitney dual for  $\mathcal{P}_n$ in next section.

We let $\Lambda$ denote the poset structure on $[n] \times [n]$ given by $(u',\ell') \leq (u, \ell)$ if $u \leq u^\prime$ and $\ell \leq \ell^\prime$. 
If $M =  M[U,L]$ covers $M' = M[U',L']$ in $\mathcal{P}_n$, we label this cover relation by the good pair $(U\setminus U',L\setminus L')\in [n]\times [n]$. 
This defines an edge labeling $\lambda:E(H({\mathcal P}_n)) \rightarrow \Lambda_n$, which we call the \emph{good pair} labeling of ${\mathcal P}_n$.

\begin{theorem}\label{thm:EL}
The good pair labeling $\lambda$ provides an EL-labeling of ${\mathcal P}_n$.  
\end{theorem}
\begin{proof}
Let $M'=M[U',L']\leq_q M=M[U,L]\in {\mathcal P}_n$ and let $C$ be the chain in $[M',M]$ such that $\lambda(C)=((\ell_{i_1},u_{j_1}), \ldots, (\ell_{i_z},u_{j_z}))$ is the greedy pairing on $(U\setminus U')\times (L\setminus L')$. 
From \cite[Lemma 16]{BenKna}, we see that such a chain exists.
We show that $C$ is the chain with the properties required in Definition~\ref{def:EL}. Note that for $1\leq r\leq z$ the pair  $\lambda(C)_r$ is the minimum element of $(U\setminus U'\setminus \{u_{j_1}, \ldots, u_{j_{r-1}}\})\times (L\setminus L'\setminus\{\ell_{i_1}, \ldots, \ell_{i_{r-1}}\})$ with respect to $\Lambda$. (*)

Clearly, $(\ell_{i_r},u_{j_r})\leq_{\Lambda}(\ell_{i_s},u_{j_s})$ for all $1\leq r\leq s\leq z$. Hence $C$ is weakly increasing. Suppose $C' \neq C$ is any other maximal chain in $[M',M]$, and let $r$ be the smallest index such that $\lambda(C')_r\neq \lambda(C)_r$. We have $\lambda(C)_r<_{\Lambda} \lambda(C')_r$ by (*) and $C <_L C'$. At the same time, we must have some later $s>r$ where $\lambda(C')_s=(\ell_{i_r},u)$ or $\lambda(C')_s=(\ell,u_{i_r})$. Therefore, we conclude that $\lambda(C')_r\not\leq_{\Lambda}\lambda(C')_s$ and $C'$ is not increasing. Now the claim follows.
\end{proof}

\begin{remark}
When we define LPMs, the data is given by a pair of noncrossing monotone lattice paths $U$ and $L$. We want to point out that the good pairing condition implies that $U$ and $L$ do not cross when we take quotients which can be proved by contradiction.
\end{remark}

\begin{figure}[htp]
    \centering
    \includegraphics[scale=0.5]{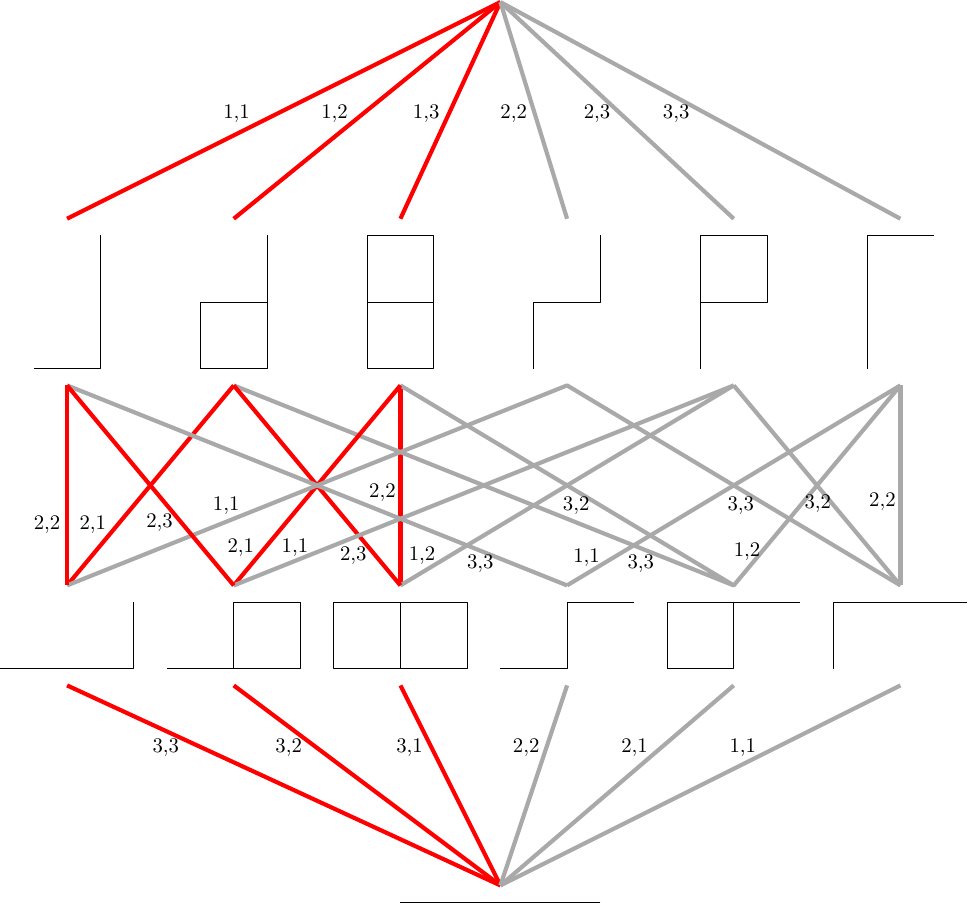}
    \caption{The EL-labeling of ${\mathcal P_3}$, along with six falling chains.}\label{fig:fallingchains}
\end{figure}

\subsection{M\"{o}bius functions}

As we have seen, for a poset with EL shelling $\lambda$, the value of the M\"{o}bius function on an interval can be recovered by counting the number of falling chains in that interval \cite[Theorem~3.4]{BjoWac}. In this section, we provide combinatorial interpretations for these falling chains in order to compute bounds for the M\"obius function.

We first fix our notation. For a maximal chain $m = ((l_1,u_1),(l_2,u_2),\dots,(l_n,u_n))$, we define an ordered pair $(\sigma, \tau)$ of permutations in $S_n$ such that $\sigma(i) = l_i$ and $\tau(i) = u_i$ for all $i$. The maximal chain $m$ can be written as a $2 \times n$ matrix
\[\begin{bmatrix}
    \ell_1 & \ell_2 & \cdots & \ell_{n-1} & \ell_n \\
    u_1 & u_2 & \cdots & u_{n-1} & u_n
  \end{bmatrix}_{\textstyle \raisebox{2pt}{.}}
\]

\begin{definition}
The \emph{Lehmer code} $L(w)$ of $w\in S_n$ is the $n$-tuple 
\[(L_1(w),L_2(w),L_3(w),\dots, L_n(w))\]
where $L_j(w) = |\{k > j: w(k) < w(j)\}|$. We also denote \[L'_j(w) := L_j(w^{-1}) = |\{k > j: w^{-1}(k) < w^{-1}(j)\}|\] and $L'(w) = (L'_1(w),L'_2(w),L'_3(w),\dots, L'_n(w))$.
\end{definition}

Clearly, we have $l(w) = \sum_{j=1}^n L_j$ and $|\{k < j: w(k) < w (j)\}| = w(j)- 1 - L_j(w)$. We make the following observation.

\begin{lemma}\label{lem:Lehmer}
For any permutation $w \in S_n$, we have $L_j(w) - L'_{w(j)}(w) = w(j) - j$ for all $j$.
\end{lemma}
\begin{proof}
Note that
\begin{align*}
    L'_j(w) &= |\{k > j: w^{-1}(k) < w^{-1}(j)\}|\\
            &= |\{k < w^{-1}(j): w(k) > j \}|\\
            &= w^{-1}(j) - 1 - |\{k < w^{-1}(j): w(k)< j)\}|.
\end{align*}
Since the equations are true for all $j$, we can substitute $j$ by $w(j)$ and then we get $|\{k < j: w(k) < w (j)\}| = j - 1 - L'_{w(j)}(w) = w(j) -1 - L_j(w).$
\end{proof}

\begin{example}
As an example, the element $w = (3 \; 1 \;5 \;4 \; 2) \in S_5$ has Lehmer code $L(w) = (2,0,2,1,0)$. Note that $w^{-1} = (2 \; 5 \; 1 \; 4 \; 3)$, so that $L^\prime(w) = (1,3,0,1,0)$. Note that $L_1(w) - L^\prime_3(w) = 2 - 0 = 2 = 3 - 1 = w(1) - 1$.
\end{example}

We next translate both the good pairing condition in Definition \ref{def:good_pair} and the falling chain condition into the language of permutations. Suppose $m$ is a maximal chain in ${\mathcal P}_n$ and let $\sigma, \tau \in S_n$ be the corresponding permutations defined above. Then the falling chain condition $(l_i, u_i) \nless (l_{i+1}, u_{i+1})$ is equivalent to saying that $A(\sigma) \cap A(\tau) = \emptyset$, where the \emph{ascent set} $A(\sigma)$ is defined as $A(\sigma) = \{i : \sigma(i) < \sigma(i+1)\} \subseteq [n-1]$. 

Now, to translate the good pairing condition in terms of $\sigma$ and $\tau$, we suppose $((l_1,u_1),(l_2,u_2),\dots,(l_n,u_n))$ is a maximal chain of good pairings. For all $j \in [n]$, we have
\begin{equation}\label{good_pairing_max_chain}
    \begin{split}
    |\{k > j: \sigma(k) < \sigma(j)\}| &\geq |\{k > j: \tau(k) < \tau(j)\}| \\
    \sigma(j) -  |\{k > j: \sigma(k) < \sigma(j)\}| &\leq \tau(j) - |\{k > j: \tau(k) < \tau(j)\}|
    \end{split}
\end{equation}

From Lemma \ref{lem:Lehmer} we have $w(j) - L_j(w) = j - L'_{w(j)}(w)$ for any $w \in S_n$. Hence we can rewrite the equation (\ref{good_pairing_max_chain}) as follows. For all $j \in [n]$, we have
\begin{equation}\label{def:good_pair_perm}
    \begin{split}
    L_j(\sigma) &\geq L_j(\tau)  \\
    L'_{\sigma(j)}(\sigma) &\geq L'_{\tau(j)}(\tau) 
    \end{split}
\end{equation}
\begin{definition}
We say a pair of permutations $(\sigma,\tau)\in S_n \times S_n$ is a \emph{good pair} if it satisfies equations (\ref{def:good_pair_perm}). We denote the set of all good pairs of permutations in $S_n$ as $GP(n)$.
\end{definition}

This discussion, along with Theorem \ref{thm:falling}, then proves the following.

\begin{theorem} \label{thm:goodpairs}
$\mu(U_{0,n},U_{n,n}) = |\{(\sigma,\tau) \in GP(n): A(\sigma)\cap A(\tau) = \emptyset\}|$.
\end{theorem}

\subsection{Counting falling chains}

The goal of this subsection is to count falling chains in the interval $(U_{0,n},U_{n,n})$ under the EL-labeling described above. We first observe a kind of duality of falling chains.

\begin{lemma}\label{lem:dual}
    If $m = ((l_1,u_1),(l_2,u_2),\dots,(l_n,u_n))$ is a falling chain, then 
    \[m^* = ((n+1-l_n, n+1-u_n), \dots, (n+1-l_1, n+1-u_1))\]
    \noindent
    is also a falling chain in $\mathcal{P}_n$.
\end{lemma}

\begin{proof}
    Since $(l_i, u_i) \nless (l_{i+1},u_{i+1})$, we have $(n+1-l_{i+1}, n+1-u_{i+1})\nless (n+1-l_i,n+1-u_i)$. It suffices to show that $m^*$ is a chain in $\mathcal{P}_n$. We will use Lehmer code to prove this statement. If $(w,v)\in GP(n)$ gives the maximal chain $m$, then $m^*$ is given by $(\sigma w \sigma, \sigma v \sigma)\in S_n \times S_n$ where $\sigma$ represents the longest permutation in $S_n$. By the definition of Lehmer code, we have for all $j$ 
    \begin{equation}
    \begin{split}
    L_j(\sigma w\sigma) &= |\{k>j| w(n+1-k) > w(n+1-j)\}| \\
    &= |\{k < n+1-j| w(k) > w(n+1-j)\}| \\
    &= (n+1-j) - w(n+1-j) +1 +L_{n+1-j}(w).
\end{split}
\end{equation}
    
    Thus, we conclude $L_j(\sigma w\sigma) \geq L_j(\sigma v\sigma)$. We can also compute $\sigma w\sigma(j) - L_j(\sigma w\sigma) = [n+1 - w(n+1-j)] - [(n+1-j) - w(n+1-j) +1 +L_{n+1-j}(w)] = j -1 -L_{n+1-j}(w)$. This gives us $\sigma w\sigma(j) - L_j(\sigma w\sigma) \leq \sigma v\sigma(j) - L_j(\sigma v\sigma)$ which implies $L'_{\sigma w \sigma(j)}(\sigma w \sigma) \geq L'_{\sigma v \sigma (j)}(\sigma v \sigma)$.
\end{proof}

\begin{lemma}\label{lem:l_n=1}

Let $m = ((l_1,u_1),(l_2,u_2),\dots,(l_n,u_n))$ be a maximal chain in $\mathcal{P}_n$. If $m$ is a falling chain, then $l_n = 1$ and $l_1 = n$.
\end{lemma}

\begin{proof}
Suppose $l_n \neq 1$ and $l_i = 1$ for some $i < n$. Since $(l_i,u_i)$ is a good pairing, the vertical segment corresponding to $u_i$ must place at first column, which implies $u_k > u_i$ for all $k > i$. In particular, we have $u_{i+1} > u_i$. However, this contradicts with the falling chain condition since $(l_i, u_i) < (l_{i+1}, u_{i+1})$. The fact that $l_1= n$ follows from Lemma $\ref{lem:dual}$.
\end{proof}

\begin{definition}\label{def:Mk1}
Let $M_{k,1}$ denote the rank $n-1$ matroid $M[[n]\setminus \{k\}, [n]\setminus \{1\}]$. We denote the set of all falling chains in $[U_{0,n} , M_{k,1}]$ by $FC_k$.
\end{definition}

Note that since $l_n = 1$ for all falling chains from Lemma \ref{lem:l_n=1}, to compute all falling chains in ${\mathcal P}_n$ we only need to consider those in the interval $[U_{0,n} , M_{k,1}]$. Also note that $[U_{0,n},M_{1,1}] \cong {\mathcal P}_{n-1}$.

\begin{theorem}\label{thm:FC_a>FC_b}
    
We have $|FC_a| \geq |FC_b|$ whenever $a\geq b$.

\end{theorem}
\begin{proof}
    We show that $|FC_{a+1}| \geq |FC_a|$ for any $a$. Let $\phi_a: \mathbb{N}\setminus \{a\} \rightarrow \mathbb{N}\setminus \{a+1\}$ by $\phi_a(a+1) = a$, and $\phi_a(k) =k$ otherwise. For a maximal falling chain, $$m = ((l_1,u_1),(l_2,u_2),\dots,(l_{n-1},u_{n-1})) \in FC_a,$$ we define $\iota(m) = ((l_1,\phi_a(u_1)),(l_2,\phi_a(u_2)),\dots,(l_{n-1},\phi_a(u_{n-1})))$. Clearly, $\iota(m)$ satisfies the falling condition, so it suffices to show that $\iota(m)$ is actually a chain in $[U_{0,n} , M_{a+1,1}]$. We prove it by showing that $\iota(m)$ can be extended to a chain $$((l_1,\phi_a(u_1)),(l_2,\phi_a(u_2)),\dots,(l_{n-1},\phi_a(u_{n-1})), (1,a+1))$$ in $\mathcal{P}_n$. Since $((l_1,u_1),(l_2,u_2),\dots,(l_{n-1},u_{n-1}),(1,a))$ as the extension of $m$ in $\mathcal{P}_n$ is a maximal falling chain, we have $L_j(w) \geq L_j(v)$ and $ L'_{w(j)}(w) \geq L'_{v(j)}(v)$ for $(w,v)\in S_n\times S_n$ corresponding to this chain. Define $v'\in S_n$ such that $v'(i) = \phi_a(u_i)$ for $i<n$ and $v'(n) = a+1$. By a straight computation, we have $L_j(v) \geq L_j(v')$ and $L'_{v(j)} = L'_{v'(j)}$. Thus, the extension of $\iota(m)$ given by the pair $(w,v')$ is a maximal chain in $\mathcal{P}_n$.
\end{proof}

In what follows, we let ${\mathcal F}_n$ denote the set of maximal falling chains in ${\mathcal P}_n$, under the EL-labeling described above. Then we have the following observation.
\begin{corollary}\label{cor:lowerbound}
For $n \geq 2$, we have $|{\mathcal F}_{n+1}| \geq (n+1) |{\mathcal F}_n|$.
   \end{corollary}
\begin{proof}
We have the following relations
\[
|\mathcal{F}_{n+1}| = \sum_{a=1}^{n+1}|FC_a| \geq (n+1)|FC_1| = (n+1)|\mathcal{F}_n|
\]
where the last equality follows from the fact that the interval $[U_{0,0},M_{1,1}]$ in $\mathcal{P}_{n+1}$ is isomorphic to $\mathcal{P}_n$.
\end{proof}

We were able to compute the $|{\mathcal F}_n|$ for $n \leq 9$. We have listed these values below, along with the set of all maximal chains in ${\mathcal P}_n$, which correspond to the set of full lattice path flag matroids, denoted LPFMs.
The number of full positroid flag matroids (PFMs) is listed in the last row. None of these sequences appears in the OEIS.

\begin{center}
\begin{tabular}{ c  | c c c c c c c c c c}
 \text{$n$}     & 1 & 2 & 3 & 4 & 5 & 6 & 7 & 8 & 9 \\
 \hline
 \text{falling} & 1 & 2 & 6 & 25 & 140 & 1031 & 9784 & 117212 & 1737600 \\
 \text{all LPFM} & 1 & 3 & 17 & 152 & 1949 & 33774 & 759391 & 21499164 & ? \\
  \text{all nonnegative} & 1 & 3 & 19 & 232 & 4013 & 102420 & 3653339 & ? & ? \\
\end{tabular}
\end{center}

We next wish to enumerate elements in ${\mathcal F}_n$ according the structure of the underlying permutations.
As we have seen, a maximal chain in $\mathcal{P}_n$ gives rise to an ordered pair of permutations $(\sigma, \tau)$.   Our goal is to determine which pairs of permutations arise this way. Given a pair $(\sigma, \tau)$ of permutations in $S_n$, we let 
\[C_{(\sigma,\tau)}  = \big((\sigma(1),\tau(1)), (\sigma(2), \tau(2)), \dots (\sigma(n),\tau(n))\big)\]
\noindent
denote the (potential) chain in the poset ${\mathcal P}_n$.

We fix $\sigma \in S_n$ and denote
\[C_\sigma = \{C_{(\sigma,\tau)} \mid C_{(\sigma,\tau)} \text{ is a falling chain of } \mathcal{P}_n \text{ for some } \tau \in S_n\}.
\]

\begin{lemma}\label{lem:sigma=w_0}
    Let $w_0$ represent the longest permutation in the symmetric group. Then the chains $C_{(w_0,\tau)}$ are falling chain of ${\mathcal P}_n$ for all $\tau\in S_n$.
\end{lemma}

\begin{proof}
    It's clear that $C_\tau$ is falling. To see that $C_\tau$ defines a chain in ${\mathcal P}_n$, note that when removing pairs from top to bottom, as an invariant we have that $L$ is all-east and then all-north, and the next element to be removed from it is the first north step. Hence with whatever element to be removed from $U$ it will form a good pair.
\end{proof}

Note in ${\mathcal P}_3$, the $6$ falling chains depicted in Figure \ref{fig:fallingchains} are all of the form $C_{(w_0, \tau)}$, where $\tau$ varies over all elements in $S_3$. For example the falling chain $((3,2),(2,1),(1,3))$ corresponds to $\tau = (2\,1\,3)$.

Lemma \ref{lem:sigma=w_0} also implies that $|{\mathcal F}_n| \geq n!$ for any $n$. This also follows from (the stronger) Corollary \ref{cor:lowerbound}, along with the observation that $|{\mathcal F}_1| =1$.

\begin{example} \label{ex:chains}
In $\mathcal{P}_4$ we have $25$ falling chains (so that $|{\mathcal F}_4| = 25$), $24$ of which follow from Lemma~ \ref{lem:sigma=w_0}. There is one \emph{exceptional} chain given by $C = \big((4, 1),(2, 3),(3, 2),(1,4)\big)$. Similarly we have $|{\mathcal F}_5| = 140$, of which $120$ come from Lemma  \ref{lem:sigma=w_0}. Recall that $[U_{0,5},M_{1,1}] \cong {\mathcal P}_4$, and hence $C$ can extended to the falling chain $C^\prime = ((5,2), (3,4), (4,3), (2,5), (1,1))$ in ${\mathcal P}_5$. As in the proof of Theorem \ref{thm:FC_a>FC_b} we can also extend $C^\prime$ to falling chains in  $[U_{0,5},M_{k,1}]$ for $k = 2, \dots, 5$. We list these below:
\begin{align*}  
&((5,2), (3,4), (4,3), (2,5), (1,1)),\\
&((5,1), (3,4), (4,3), (2,5), (1,2)),\\
&((5,1), (3,4), (4,2), (2,5), (1,3)),\\
&((5,1), (3,3), (4,2), (2,5), (1,4)),\\
&((5,1), (3,3), (4,2), (2,4), (1,5)).\\
\end{align*}
\end{example}
At the other extreme, we have the following.

\begin{lemma}\label{lem:tau=id}
    Let $\tau = (1 \; 2 \; \dots \; n)$ denote the identity permutation, and let $\sigma$ be any permutation. Then $C_{(\sigma, \tau)}$ is not a falling chain in ${\mathcal P}_n$ unless $\sigma = w_0$.
\end{lemma}

\begin{proof}
If $\sigma \neq w_0$, one can check that $C_{(\sigma,\tau)}$ has an increasing step.
\end{proof}

We next observe a kind of duality in the set of falling chains in ${\mathcal P}_n$.

\begin{proposition}\label{prop:duality}
    Suppose $\sigma$ and $\sigma^\prime$ are permutations of $[n]$ with reduced words expressions $\sigma = s_{k_1}s_{k_2}\dots s_{k_l}$ and $\sigma^\prime = s_{n-{k_1}}s_{n-{k_2}}\dots s_{n-{k_l}}$. Then the number of falling chains of the form $C_{(\sigma,\tau)}$ is equal to the number of falling chains of the form $C_{(\sigma^\prime, \tau)}$.  
\end{proposition}

\begin{proof}
    We define a group automorphism $\phi: S_n \rightarrow S_n$ by conjugating the longest permutation $\phi(\sigma) = w_0\sigma w_0$. Since $\phi(s_i) = s_{n-i}$, we have $\phi(\sigma) = \sigma^\prime$ from the statement. Actually, the converse is also true. If $\sigma^\prime = \phi(\sigma)$ and $\sigma = s_{k_1}s_{k_2}\dots s_{k_l}$ is a reduced word for $\sigma$, then $s_{n-{k_1}}s_{n-{k_2}}\dots s_{n-{k_l}}$ is a reduced word for $\sigma^\prime$. By the above relation between $\sigma$ and $\sigma^\prime$, we have $\sigma(i) = n + 1 - \sigma^\prime(n+1-i)$. Thus, by Lemma \ref{lem:dual}, we have a bijection between $C_\sigma$ and $C_\sigma^\prime$.
\end{proof}
Lemmas \ref{lem:l_n=1} and \ref{lem:sigma=w_0} and Proposition \ref{prop:duality} suggest that the more inversions a permutation $\sigma$ has, the more falling chains we have in $C_\sigma$. Recall that the \emph{left Bruhat order} $\preceq_L$ on $S_n$ by for $u \preceq_L w$ if $w = s_{i_1}\cdots s_{i_k}u$, such that $ l(s_{i_1}\cdots s_{i_k}u) = l(u)+k$, for $w,u\in S_n$. 
Note that in this order we have the minimal element given by the identity permutation and the maximal element given by $w_0$. Also note that the $\sigma$ and $\sigma^\prime$ from the Proposition \ref{prop:duality} have the same rank in the Bruhat lattice. The above results suggest that moving $\sigma$ down the Bruhat order limits the number of falling chains one has in $C_\sigma$. Indeed this is the case.

\begin{theorem}\label{thm:Bruhat}
   Suppose $\sigma$ and $\sigma'$ are elements in $S_n$ with $\sigma \preceq_L \sigma'$. Then we have $|C_\sigma| \leq |C_{\sigma'}|$. 

\end{theorem}

\begin{proof}
    We only need to prove the statement is true when $\sigma \lessdot_L \sigma'$ which means $\sigma' = s_i\sigma$. Since $l(\sigma') = l(\sigma) + 1$, we must have $\sigma^{-1}(i) < \sigma^{-1}(i+1)$. Given a falling chain $C_{(\sigma,\tau)}$, we want to show $C_{(\sigma',\tau)}$ is also a falling chain in $\mathcal{P}_n$. We first check the falling condition. If $(\sigma,\tau)$ satisfies the falling condition, we want to show $(\sigma',\tau)$ also satisfies the falling condition. Left multiplication to a permutation is just switching the position of number $i$ and $i+1$ in one line notion. If $\sigma = (\dots,i,*,\dots,*,i+1,\dots)$, then $(\sigma',\tau)$ will also satisfy the falling condition since switching position of $i$ and $i+1$ will not affect the comparability and relation in our edge-labeling of the chains. Now we start checking the good pairing condition. We just need to show that the chain $C_{(\sigma',\sigma)}$ satisfies the good pairing condition. Let $L(\sigma) = (a_1,a_2,...a_{n-1},a_n)$ represents the Lehmer code for $\sigma$. By definition $a_j = |\{k > j: \sigma(k) < \sigma(j)\}$, we can conclude that $L_j(\sigma') = L_j(\sigma) = a_j$ when $j\neq \sigma^{-1}(i)$, and $L_{\sigma'^{-1}(i+1)}(\sigma) = L_{\sigma^{-1}(i)}+1 = a_{\sigma^{-1}(i)}+1$. Thus, we have $L_j(\sigma')\geq L_j(\sigma)$ for all $i$. We also have $\sigma(j) - L_j(\sigma) = \sigma'(j) - L_j(\sigma')$ for all $j\neq \sigma^{-1}(i+1)$ and $\sigma(j) - L_j(\sigma) +1 = \sigma'(j) - L_j(\sigma')$ for $j = \sigma^{-1}(i+1) = \sigma'^{-1}(i)$. Thus, we get $L'_{\sigma(j)}(\sigma) \geq L'_{\sigma'(j)}(\sigma')$. Thus, we prove that $(\sigma',\sigma)$ satisfies the inequalities (\ref{def:good_pair_perm}).
\end{proof}

In general, it seems to difficult to determine $|C_\sigma|$ for arbitrary $\sigma \neq w_0$. However we were able to compute the following.

\begin{theorem}\label{thm:siw0}
    The number of falling chain in poset of LPMs with given lower path permutation $s_iw_0$ where $2\leq i \leq n-1$, $|C_{s_iw_0}|$, equals to 
    $$\frac{n!}{(n-i)(n-i+1)}+\frac{n!}{i(i+1)}+\frac{n!}{i(n-i)}+\frac{n!}{2}-\frac{n!}{i}-\frac{n!}{n-i}-(i-1)!(n-1-i)!$$
    
\end{theorem}

\begin{proof}
    Let $u = (u_1,u_2,\dots,u_{n-i-1},u_{n-i},u_{n-i+1},u_{n-i+2},\dots, u_n)$ represent the one line notation of an element in $S_n$. We fix $i$ and consider the following $3$ sets:
    \begin{enumerate}
        \item $A = \{u \mid u_{n-i} < u_{n-i+k} \text{ for some } k\geq 2\}$
        \item $B = \{u \mid u_{n-i+1} > u_{n-i+1-k} \text{ for some } k\geq 2\}$
        \item $C = \{u \mid u_{n-i} > u_{n-i+1}\}$
    \end{enumerate}
    By the definition of falling chains in LPMs, the set $A \cap B \cap C$ represents all falling chains with fixed lower path permutation $s_iw_0$. We would compute the cardinality of $A \cap B \cap C$ and we use inclusion and exclusion principle to compute it. We will compute the cardinality of the following sets:
    \begin{enumerate}
        \item $|A| = n! - {{n} \choose {i}}(i-1)!(n-i)! = n! - \frac{n!}{i}$ where ${{n} \choose {i}}(i-1)!(n-i)! = |A^c|$. Note that the size of $A^c$ is computed by picking $i$ numbers from $[n]$ for $u_{n-i},u_{n-i+2},u_{n-i+3},\dots ,u_n$. Then we assign the largest picked number to $u_{n-i}$ and permute the rest of numbers. 
        \item $|B| = n!-\frac{n!}{(n-i)}$.
        \item $|C| = \frac{n!}{2}$ 
        \item $|A \cap B| = |A|+|B|-|A\cup B| = 2n! -\frac{n!}{i}-\frac{n!}{n-i} - (n! - \frac{n!}{i(n-i)})$ where $\frac{n!}{i(n-i)} = |A\cup B|^c$. The size of $(A\cup B)^c$ is computed by picking $i$ numbers from $[n]$ for $u_{n-i},u_{n-i+2},u_{n-i+3},\dots ,u_n$. Then we assign the largest picked number to $u_{n-i}$ and assign the smallest unpicked number to $u_{n-i+1}$ and permute the rest of numbers. Therefore, $|A\cup B|^c = {n\choose i}(i-1)!(n-i-1)!$.  
        \item $|A \cap C| = |A|+|C|-|A\cup C| = n!-\frac{n!}{i} + \frac{n!}{2} - (n! -\frac{n!}{i(i+1)}) = \frac{n!}{2} - \frac{n!}{i}+\frac{n!}{i(i+1)}$ where $\frac{n!}{i(i+1)} = |A\cup C|^c$. The size of $(A\cup C)^c$ is computed by picking $i+1$ numbers from $[n]$ for $u_{n-i},u_{n-i+1},u_{n-i+2},\dots ,u_n$. Then we assign the largest picked number to $u_{n-i+1}$ and second largest picked number to $u_{n-i}$. Therefore, $|A\cup C|^c = {n \choose i+1}(i-1)!(n-i-1)!$.
        \item $|B\cap C| = \frac{n!}{2} - \frac{n!}{n-i}+\frac{n!}{(n-i)(n-i-1)}$.
        \item $|A \cup B \cup C| = n!- (i-1)!(n-i-1)!$ where $(n-i-1) = |A \cup B \cup C|^c$. The size of $(A \cup B \cup C)^c$ is given by assigning $u_{n-i} = i, u_{n-i+1} = i+1$, $(u_1,u_2,\dots, u_{n-i-1})$ to a permutation of $\{i+2,i+3,\dots, n-1,n\}$ and $(u_{n-i+2},u_{n-i+3},\dots, u_n)$ to a permutation of $[i-1]$.
    \end{enumerate} 
    
    By inclusion-exclusion, we have $|A|+|B|+|C|-|A\cap B|-|A\cap C|-|B\cap C|+|A\cap B\cap C| = |A\cup B\cup C|$. We plug-in the values for sizes of those sets and obtain  the above formula.
\end{proof}

For example for $n=4$, if $i=2$ then $\sigma = s_2 w_0 = (4 \; 2 \;  3 \; 1)$ and we get from Theorem \ref{thm:siw0} that
\[|C_\sigma| = 4 + 4 + 6 + 12 - 12 - 12 - 1 = 1.\]
Indeed, we have seen in Example \ref{ex:chains} that $C_\sigma = \{C = \big((4, 1),(2, 3),(3, 2),(1,4)\big)\}$.

\section{Whitney duality}\label{sec:whitney_dual}

In this section we show that our edge labeling of $\mathcal{P}_n$ satisfies further properties related to a notion of poset duality. To motivate these constructions, we recall the notion of Whitney numbers of ranked posets. Suppose $P$ is a poset with rank function $\rho: P \rightarrow {\mathbb N}$ and minimum element $\hat{0}$, with rank and characteristic polynomials given by
\begin{align*}
    f_P(t)=\sum_i a_it_i &\qquad \chi_P(t)=\sum_ib_it^i
\end{align*}
respectively. Here $a_i=|\rho^{-1}(i)|$ and $b_i=\sum_{x\in \rho^{-1}(i)} \mu(\hat{0},x)$ are the \emph{Whitney numbers} of second and first kind, respectively. For instance, the rank and characteristic polynomials of the poset ${\mathcal P}_3$ (see Figure \ref{fig:fallingchains}) are given by
\begin{align*}
    f(t)=1+6t+6t^2+t^3 &\qquad \chi(t)=1-6t+11t^2-6t^3.
\end{align*}

Following~\cite{GH}, we say that a poset $Q$ with minimum element $\hat{0}$ is a \emph{Whitney dual} of $P$ if $Q$ satisfies:
\begin{align*}
    f_Q(t)=\sum_i|b_i|_it_i &\qquad \chi_Q(t)=\sum_ic_it^i
\end{align*}
where $|c_i|=a_i$. 

In general it is not clear if a given $P$ has a Whitney dual $Q$, or how to construct $Q$ if one does exists. To this end D'Le\'{o}n and Hallam introduced the following notion in~\cite{GH}.
\begin{definition}\cite[Definition 3.26]{GH} \label{defn:EW}
A labeling $\lambda$ of a graded poset $P$ is called EW if it satisfies:
\begin{enumerate}
\item $\lambda$ is an \emph{$ER$-labeling}: in each closed interval $[x, y]$ of $P$ there is a unique increasing maximal chain.
\item $\lambda$ has the \emph{rank two switching property}: if for every maximal chain $(0<c_1<...<c_k<1)$ with an ascending step $\lambda(c_i<c_{i+1})<\lambda(c_{i+1}<c_{i+2})$ there is a unique element $c'_{i+1}$ such that  $\lambda(c_i<c'_{i+1})=\lambda(c_{i+1}<c_{i+2})$ and $\lambda(c_i<c_{i+1})=\lambda(c'_{i+1}<c_{i+2})$ and $(0<c_1<...<c_i<c'_{i+1}<c_{i+2}<...<c_k<1)$ is a chain.
    \item in each interval, each maximal falling chain has a unique word of labels\footnote{Note that in \cite{GH} a falling chain is called \emph{ascent-free}.}.
\end{enumerate}
\end{definition}

We are able to show that $\mathcal{P}_n$ admits an EW-labeling and hence has a Whitney dual.

\begin{proposition}\label{prop:EW}
    The good pair labeling is an EW-labeling of $\mathcal{P}_n$.
\end{proposition}
\begin{proof}
Property (1) follows since our labeling is an EL-labeling from Theorem~\ref{thm:EL}. Property (2) is an immediate consequence of Lemma~\ref{lem:removing}. Finally, Property (3) is satisfied trivially because our $\lambda$ is injective on the chains, i.e., no two saturated chains have the same sequence of labels.
\end{proof}

In \cite{GH}, the authors describe how to construct a Whitney dual $Q_{\lambda}(P)$ of a poset $P$ with a given EW-labeling $\lambda$. The basic idea is to consider the poset consisting of all saturated chains containing $\hat{0}$, and to construct a quotient defined by an equivalence relation determined by $\lambda$. We refer to \cite{GH} for details, and in what follows we describe the construction for the case of ${\mathcal P}_3$.

For this, first recall that any Whitney dual of ${\mathcal P}_3$ is a poset whose rank function is given by $1+6t+11t^2+6t^3$.
Let $C_3$ be the poset whose elements are saturated chains in ${\mathcal P}_3$ containing $\hat{0}$, ordered by containment. Since ${\mathcal P}_3$ has 17 maximal chains (see Figure~\ref{fig:Q_3}), $C_3$ is a poset with $17$ maximal elements illustrated in the top of Figure \ref{fig:C_3}. The edges of the Hasse diagram of $C_3$ are labeled using good pairs coming from our labeling. In Figure \ref{fig:C_3} the reader can see the poset $C_3$ with its 17 maximal chains.

\begin{figure}[htp]
    \centering
    \includegraphics[width=12cm]{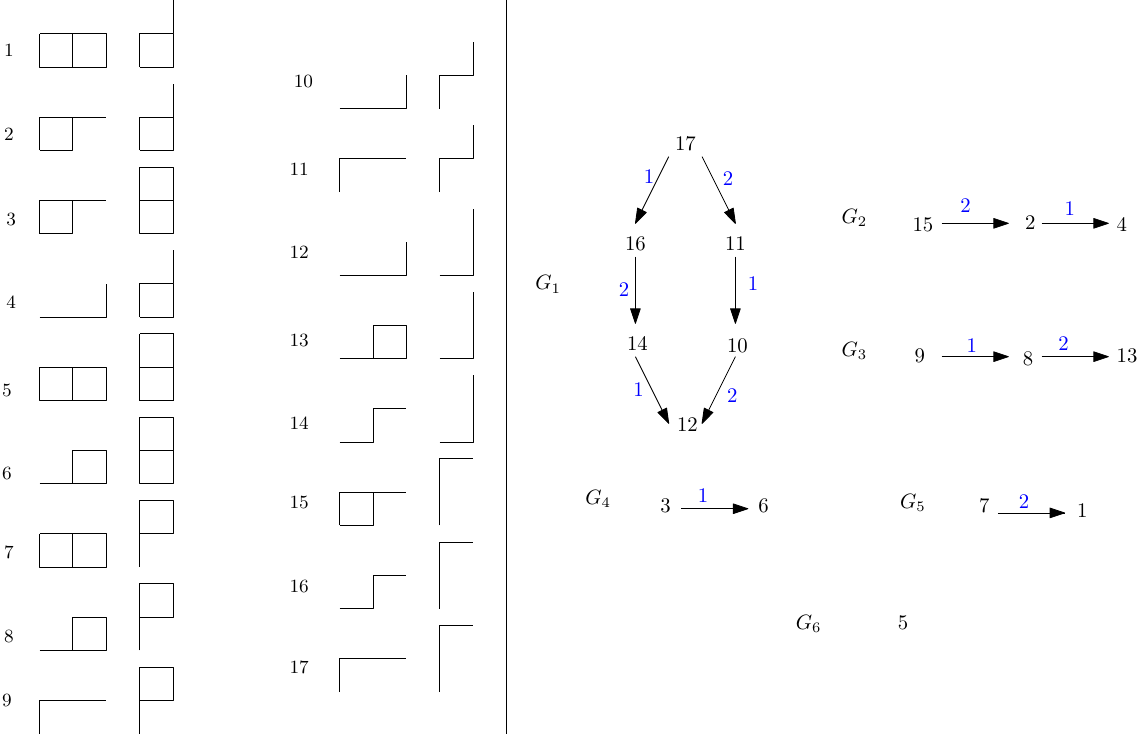}
    \caption{Left: the 17 maximal chains of the poset $P_3$. Right: The Whitney equivalence classes partitioned into 6 graphs. Blue labels on edges correspond to equivalences in $C_3$.}\label{fig:Q_3}
\end{figure}

In the poset $C_3$ we say that two elements $c:\hat{0} = a_0\lessdot a_1\lessdot\cdots\lessdot a_k$ and $c':\hat{0} = b_0\lessdot b_1\lessdot\cdots\lessdot b_k$ are equivalent, denoted $c\sim c'$, if there exists $i$ such that $\lambda(c)_i=\lambda(c')_{i+1}$, $ \lambda(c)_{i+1}=\lambda(c')_i$ and $\lambda(c)_j=\lambda(c')_j$ for all other $j$, where $a_i\xrightarrow{\lambda(c)_i}a_{i+1}$, and similarly for $c'$. For instance, in $C_3$ we have $c_3\sim c_6$ since the sequence of labels of $c_3$ are $(21,32,13)$ and in $c_6$ we have $(32,21,13)$. We then write $c_3\xrightarrow{1}c_6$ as their labels $1$ and $2$ are swapped. In this way each equivalence class carries the structure of a directed graph with unique sink as shown on the right of Figure \ref{fig:Q_3}. We obtain a poset $\mathcal{Q}_3$ by identifying the equivalence classes of $\sim$; see the bottom of Figure \ref{fig:C_3}. The rank function and characteristic polynomial of $\mathcal{Q}_3$ are $1+6t+11t^2+6t^3$ and $1-6t+6t^2-t^3$, respectively. Hence, $\mathcal{Q}_3$ is a Whitney dual of $\mathcal{P}_3$.

\begin{figure}[htp]
    \centering
    \includegraphics[width=12cm]{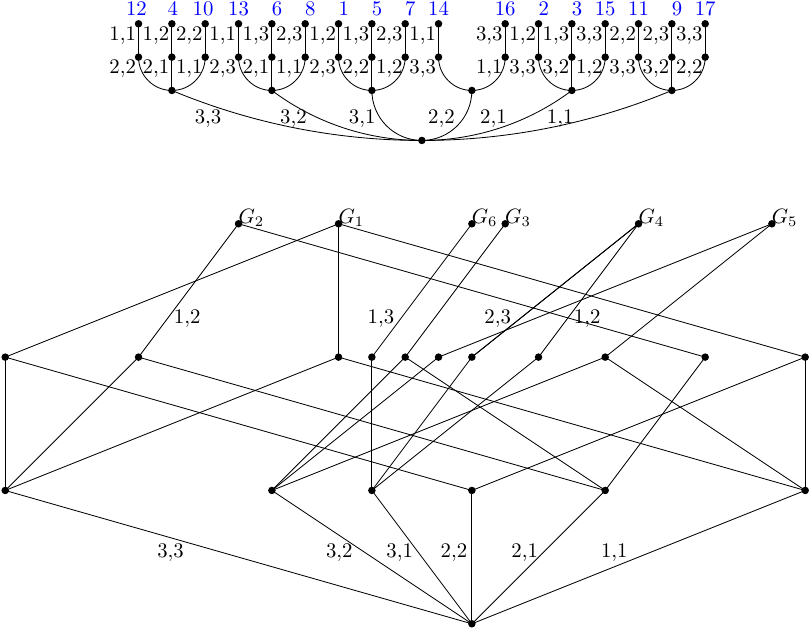}
    \caption{Top: $C_3$. Bottom: Whitney dual ${\mathcal Q}_3$ of ${\mathcal P}_3$. To reduce notation, edges with the same label are drawn with the same slope.}\label{fig:C_3}
\end{figure}

\begin{remark}
We note that since our labeling is both EW and EL, the number of falling chains equals the number of equivalence classes of saturated maximal chains, as defined in the construction of ${\mathcal Q}_3$ above. Both are counted by $\mu(\hat{0},\hat{1})$. Indeed, as displayed in Figure \ref{fig:C_3}, the sinks of $G_1,\dots, G_6$ are $12,4,13,6,1,5$, respectively. Each such sink corresponds to each of the falling chains from Figure \ref{fig:fallingchains}.

 \end{remark}

\section{Further questions}\label{sec:further}
In this section we collect some open questions and directions for further study.

\begin{question}\label{ques:falling}
For our EL-labeling of ${\mathcal P}_n$, can one describe better bounds for the number of falling chains in any interval?
\end{question}

Recall that an answer to Question \ref{ques:falling} would provide bounds on the values of the M\"obius function for ${\mathcal P}_n$. In particular $\mu(\hat 0, \hat 1)$ is given by the number of falling chains in the poset ${\mathcal P}_n$ itself.

Recall from the introduction that it is of interest to study certain positroid flags ${\mathcal M}: M_0 \leq_q M_1 \leq_q \cdots \leq_q M_n$, where each $M_i$ is a quotient of $M_j$ for $i < j$. Let ${\mathcal O}_n$ denote the poset of positroid flags on ground set $[n]$ (see \cite{BChT}). Note that ${\mathcal P}_n$ is an induced (full?) subposet  of ${\mathcal O}_n$. Recall that a maximal chain in ${\mathcal O}_n$ corresponds to a point in the nonnegative flag variety if and only if the underlying flag matroid polytope $P_{\mathcal M}$ is an interval in the Bruhat order of $S_n$.   We have also seen that a lattice path flag matroid (by definition a maximal chain in ${\mathcal P}_n$) is an example of such a point.

We have seen that ${\mathcal P}_n$ has the homotopy type of a wedge of spheres, and we also have the inclusion $\iota: {\mathcal P}_n \rightarrow {\mathcal O}_n$.
It then natural to ask if this inclusion relates the topology of the two posets.

\begin{question}
Can one understand the topology of the map $\iota: {\mathcal P}_n \rightarrow {\mathcal O}_n$? In particular, can one construct a retract $r: {\mathcal O}_n \rightarrow {\mathcal P}_n$ whose fibers are highly connected?
\end{question}

Note that for any positroid $P$ one can naturally define an LPM $M_P[U,L]$ by taking $U$ (respectively $L$) to be minimal basis (resp. maximal) in Gale order. Does this define a poset map? 

\section{Acknowledgments}
We thank Bruce Sagan for asking the question that inspired this project, and also Rafael S. Gonaz\'alez D'Le\'on for helpful conversations. CB thanks CIMPA-ICTP Research in Pairs Fellowship. AD was partially supported by Simons Foundation Grant $\#964659$.  KK was partially supported by the Severo Ochoa and Mar\'ia de Maeztu Program for Centers and Units of Excellence in R\&D (CEX2020-001084-M) and the Spanish \emph{Ministerio de Econom\'ia, Industria y Competitividad}
through grant PID2022-137283NB-C22.

\bibliographystyle{amsplain}
\bibliography{main.bib}

\end{document}